\documentclass[12pt]{amsart} 
\usepackage{amsmath,amssymb,latexsym,amsthm,enumerate,amscd} 
\usepackage[abbrev, bibtex-style]{amsrefs} 

\BibSpec{misc}{%
  +{}{\PrintAuthors} {author}
  +{,}{ \textit} {title}
  +{}{ \parenthesize} {date}
  +{,}{ \url} {url}
  +{,}{ } {note}
  +{.}{ } {transition}
}

\usepackage[nocolor]{sseq} 
\usepackage[textwidth=1.2in, textsize=small]{todonotes} 
\usepackage{tikz, tikz-cd} 
\usetikzlibrary{arrows} 
\usetikzlibrary{matrix} 
\usetikzlibrary{shapes} 
\usepackage{microtype} 
\usepackage{geometry}

\usetikzlibrary{calc} 
\usepgflibrary{shapes.geometric}  
\renewcommand{\d}{\partial }

\newtheoremstyle{dotless}{}{}{\itshape}{}{\bfseries}{}{ }{}

\DeclareMathOperator\stab{St}

\newtheorem{Theorem}{Theorem}[section]

\newtheorem{Corollary}[Theorem]{Corollary} 
\newtheorem{Lemma}[Theorem]{Lemma}

\newtheorem*{Conjecture1}{Singer conjecture} 
\newtheorem*{Conjecture2}{Action dimension (actdim) conjecture} 
\newtheorem*{Conjecture3}{Cocompact action dimension (cadim) conjecture} 
\newtheorem*{Cadimm}{Cadim conjecture in dimension $n$} 
 
\newtheorem*{Conjecture5}{Weak cadim conjecture} 

\theoremstyle{definition} 
\newtheorem*{Definition}{Definition} 

\theoremstyle{remark} 
\newtheorem*{Remark}{Remark}

\newcommand{\cE}{\mathcal{E}}
\newcommand{\cU}{\mathcal{U}}
\DeclareMathOperator\cadim{cadim}
\DeclareMathOperator\actdim{actdim}

\begin{document}

\author{Boris Okun} 
\email{okun@uwm.edu} 
\author{Kevin Schreve} 
\email{kschreve@uwm.edu}

\title{The $L^2$-(co)homology of groups with hierarchies} 

\date{} 
\begin{abstract}
	We study group actions on manifolds that admit hierarchies, which generalizes the idea of Haken $n$-manifolds introduced by Foozwell and Rubinstein.
	We show that these manifolds satisfy the Singer conjecture in dimensions $n \le 4$.
	Our main application is to Coxeter groups whose Davis complexes are manifolds; we show that the natural action of these groups on the Davis complex has a hierarchy.
	Our second result is that the Singer conjecture is equivalent to the cocompact action dimension conjecture, which is a statement about all groups, not just fundamental groups of closed aspherical manifolds.
\end{abstract}

\maketitle

\section*{Introduction}

In his Ph.D. thesis~\cite{f07}, Foozwell introduced Haken $n$-manifolds as a higher dimensional analogue of Haken $3$-manifolds. 
Loosely speaking, these are closed $n$-manifolds that can be cut inductively along codimension one submanifolds to a disjoint union of $n$-balls.   
The exact definition is somewhat technical. 
The resulting sequence of manifolds is called a hierarchy. 
Foozwell and Rubinstein have explored many properties of these manifolds, in particular, they have shown~\cite{f11a,fr11} that their universal covers are homeomorphic to $\mathbb{R}^n$ and their fundamental groups have solvable word problem.  
Both these properties show that Haken $n$-manifolds are a special class of aspherical manifolds \cites{d83,m90}.

The classical Euler characteristic conjecture, attributed to Hopf, predicts the sign of the Euler characteristic of a closed aspherical $2n$-dimensional manifold $M^{2n}$: $(-1)^{n}\chi(M^{2n}) \ge 0$. 
In a special case of right-angled Coxeter group manifolds, this conjecture becomes a purely combinatorial statement about flag simplicial triangulations of $(2n-1)$-spheres, known as the Charney--Davis conjecture~\cite{cd95b}.

Another classical conjecture about aspherical manifolds, the Singer conjecture, predicts that 
the reduced $L^2$-homology of the universal cover vanishes except possibly in the middle dimension. 
Since one can use $L^2$-Betti numbers to compute $\chi$, the Singer conjecture immediately implies the Euler characteristic conjecture.

Edmonds~\cite{e13} proved the Euler characteristic conjecture for closed Haken $4$-manifolds by showing that it was equivalent to the Charney--Davis conjecture for $3$-spheres, which holds true by a result of Davis and the first author~\cite{do01}, where the Singer conjecture for $4$-dimensional right-angled Coxeter group manifolds is proved.
This equivalence was extended by Davis and Edmonds~\cite{de14} to all even dimensions. 
In fact, they showed this equivalence for \emph{generalized Haken $2n$-manifolds}, where they allow the hierarchy to end in any compact, contractible manifold.

The starting point of this paper was a question of Edmonds whether the Singer conjecture holds for Haken $4$-manifolds. 

One advantage of studying homological properties of Haken $n$-manifolds is that we can ignore most of the technicalities and study a more general class of manifolds that is closer to the loose definition above. 
Since we are interested in group actions that are not free, and because we think it is simpler, we build the hierarchies out of contractible manifolds with a proper and cocompact group action.

We say a group $G$ \emph{admits a hierarchy} if it acts on a contractible manifold $M$ that can be cut inductively along codimension one contractible $G$-invariant submanifolds to a disjoint union of compact, contractible manifolds.
An example to keep in mind is $\mathbb{Z}^n$ acting on $\mathbb{R}^n$ with quotient the $n$-torus $T^n$. 
Cutting $T^n$ along $T^{n-1}$ corresponds to cutting along $\mathbb{Z}^n$-translates of $\mathbb{R}^{n-1}$ inside $\mathbb{R}^n$. 
In a similar way, hierarchies for Haken $n$-manifolds lift to our hierarchies on the universal covers.

The paper is organized as follows.
We develop a general theory of group actions with hierarchies in Section~\ref{s:hier}.
In Section~\ref{s:cox} we prove that Coxeter group manifolds admit hierarchies. 
In Section~\ref{s:l2} we briefly recall the necessary background material on $L^{2}$-homology.
Finally, in Section~\ref{s:vanish} we study various vanishing conjectures about $L^{2}$-homology.

Our first result is that the Singer conjecture holds for all groups that admit a hierarchy in dimension $4$. 
Our main application of this result is to Coxeter groups:
Theorem \ref{hierarchy} generalizes the result in \cite{do01} for right-angled Coxeter groups and a later result of T. Schroeder~\cite{s09} for even Coxeter groups. 

We also introduce the notion of \emph{cocompact action dimension} of a group --- the minimal dimension of a contractible manifold, possibly with boundary, which admits a proper cocompact action by the group. 
Our second result is that the Singer conjecture is actually a statement about all groups, not just about fundamental groups of closed aspherical manifolds. 
Namely, we show that the Singer conjecture is equivalent to the cocompact action dimension conjecture: the $L^{2}$-homology of a group vanishes above half of its  cocompact action dimension.

We are grateful to Mike Davis for sending us an early version of~\cite{de14} and several useful discussions.

\section{Hierarchies for group actions} \label{s:hier}
\begin{Definition}
	Let $G$ be a discrete group.
	A \emph{$G$-space} $M$ is a space with a geometric $G$-action.
	If $N$ is a $G$-invariant subspace of $M$ then $(M,N)$ is a \emph{pair of $G$-spaces.} 
\end{Definition}

\begin{Definition}
Let $M$ be a $G$-manifold, and $\cE= \{E_i\}_{i=0}^r$ a collection of codimension one $G$-submanifolds. 
$(M,\cE)$ is \emph{tidy} if 
	\begin{itemize}		
		\item The components of $M$ are contractible.		
		\item The components of any intersection of $E_{i}$'s are contractible.		
		\item $E_{i} \cap \partial M = \partial E_{i}$ for all $i$.
		\item $(M,\cE)$ locally looks like a real hyperplane arrangement: every point  in the interior of $M$ (boundary of $M$, resp.) has a chart which maps the $E_{i}$'s into linear codimension 1 subspaces in $\mathbb{R}^n$ ($\mathbb{R}_{+}^n$, resp.)
	\end{itemize}
\end{Definition}
In the case where $\cE$ consists of just one submanifold $F$, this definition is equivalent to requiring that $F$ is a neat submanifold and the components of both $M$ and $F$ are contractible. 
We will call such a pair $(M,F)$ a \emph{tidy pair}.

Note that since the components of $F$ are contractible, it admits a $G$-invariant collar neighborhood, which we denote $F\times I$.  
By \emph{cutting} $M$ along $F$ we mean removing the interior of this collar, so that we have a decomposition 
\[
	M = N \cup_{F \times \partial I} F\times I.
\]
where $N$ is $M$ \emph{cut-open} along $F$.  
Associated to this decomposition there is a Mayer--Vietoris sequence
\begin{equation}\label{e:MV}
	\rightarrow H_k(F \times \partial I) \rightarrow H_k(F \times I) \oplus H_k(N) \rightarrow H_k(M) \rightarrow
\end{equation}

\begin{Lemma}\label{cutting}
	If $(M,F)$ is a tidy pair and $N$ is $M$ cut-open along $F$, then the components of $N$ are contractible manifolds.
\end{Lemma}
Notice that $N$ may or may not have more $G$-orbits of components than does $M$.
\begin{proof}
	Using the above decomposition, the Van Kampen theorem implies that components of $N$ are simply connected, and the Mayer--Vietoris sequence~\eqref{e:MV} shows that $N$ is acyclic.
\end{proof}

\begin{Definition}
	An \emph{n-hierarchy} for an action of a discrete group $G$ on a manifold $M$ is a sequence 
	\[
		(M_0,F_0), (M_1,F_1), \dots , (M_m,F_m), (M_{m+1},\emptyset),
	\]
	such that 
	\begin{itemize}
		\item $M_0 = M$.
		\item $M_{m+1}$ is a disjoint union of compact, contractible $n$-manifolds.
		\item $(M_i,F_i)$ is a tidy pair for each $i$.
		\item $M_{i+1}$ is  $M_i$ cut-open along $F_i$.
	\end{itemize}
\end{Definition}
More generally, if $(M,N)$ is a $G$-pair of manifolds, we can define a \emph{hierarchy ending in N} in the same way, with the one difference being that $M_{m+1} = N$.

\begin{Definition}
	$G$ \emph{admits an $n$-hierarchy} if there exists a contractible, $n$-dimensional $G$-manifold $M$ and a hierarchy for the action.
\end{Definition}

\begin{Lemma}\label{l:stab}
	Let $G$ act on $M$ with a hierarchy, and let $M_1^0$ be a component of $M_1$. 
	Then there is an induced hierarchy for the action of $\stab_G(M_1^0)$ on $M_1^0$, where $\stab_G(M_1^0)$ is the stabilizer of $M_1^0$.
\end{Lemma}
\begin{proof}
	We claim the following sequence is a hierarchy for $M_1^0$:
\[
(M_1^0,F_1 \cap M_1^0), (  M_2 \cap M_1^0,  F_2 \cap M_1^0), \dots  (M_{m+1} \cap M_1^0,\emptyset)
\]

We have that $M_1^0$ is a $\stab_G(M_1^0)$-manifold, and by Lemma \ref{cutting} is contractible. 
Since each $F_i$ is $G$-invariant, $F_i \cap M_1^0$ is $\stab_G(M_1^0)$-invariant, and the other conditions of our hierarchy follow immediately.
\end{proof}

\begin{Lemma}\label{l:tidy}
	Let $(M,\cE)$ be tidy, and let $N$ be $M$ cut-open along $E_{0}$. 
	Then $(N, \{E_{i}\cap N\}_{i=1}^{r})$ is also tidy.         
\end{Lemma}
\begin{proof}

We first check the last two conditions of tidiness, as these help us prove the first two. 
Note that  $(M,E_{0})$ is a tidy pair.
 After cutting, the local picture is mostly preserved, we just have to check near $E_0$. 
 If $x \in (E_0 \times \partial I) - \partial M$, the new charts come from restricting the old chart to a halfspace bounded by a copy of $E_0$. 
 If $x \in \partial M \cap (E_0 \times \partial I)$ the new charts come from "straightening" $\mathbb{R}^{n-1} \cup E_0$ to $\mathbb{R}^{n-1}$ and preserving the linear structure of the other hyperplanes.
 
 Now, contractibility of the components of $N$ follows immediately from Lemma~\ref{cutting}. 
 By assumption, any intersection $\bigcap E_{i_{\alpha}}$ has contractible components, and our assumption on the local structure lets us choose a thin enough neighborhood such that the intersection between $E_0 \times I$ and $\bigcap E_{i_{\alpha}}$ deformation retracts to $\bigcap E_{i_{\alpha}} \cap E_0$. Since $N \cap E_i$ is precisely $E_i$ cut out by $E_0$, Lemma~\ref{cutting} again implies that $N \cap E_i$ has contractible components. 
 \end{proof}

\begin{Theorem}\label{tidy}
Let $M$ be a $G$-manifold, and $\cE= \{E_i\}_{i=0}^r$ a collection of  submanifolds such that $(M,\cE)$ is tidy. 
If the components of the complement $M - \cup_i E_i$ have compact closure in $M$, then the action of $G$ on $M$ admits a hierarchy.
\end{Theorem}
\begin{proof}

The proof is to apply Lemma \ref{l:tidy} repeatedly, as this implies that if we cut along each $E_i$, we get a hierarchy ending in $M - \cup_i E_i$. 
To be precise, let $F_j = E_j$ cut-along by $E_0, E_1, \dots E_{j-1}$ and let $M_0 = M$ and $M_{j+1} = M_j$ cut along by $F_j$. Since each $E_i$ is $G$-invariant, $(M_j,F_j)$ is a tidy pair for all $j$.
\end{proof}

\section{Coxeter groups}\label{s:cox}

Recall that a Coxeter group $W$ has generators $s_i$ with relations $s_i^2 = 1$ and $(s_is_j)^{m_{ij}}$ for $m_{ij} \in \mathbb{N} \cup \infty$.
 In other words, $W$ is generated by reflections and each pair of reflections generates a dihedral subgroup (possibly $D_\infty)$. 
The \emph{nerve} of a Coxeter group is a simplicial complex  with vertices corresponding to generators $s_i$, and $s_{i_1}, \dots , s_{i_n}$ a simplex iff the subgroup generated by $s_{i_1}, \dots , s_{i_n}$ is finite. 
A Coxeter group is \emph{right-angled} is $m_{ij} = 2$ or $\infty$ for all $i,j$.

\begin{Definition}
	A \emph{mirror structure} on a space $X$ is an index set $S$ and a collection of subspaces $\{X_s\}_{s \in S}$.
	For each $x \in X$, let 
	\[
		S(x) := \{s \in S \hspace{1mm}| \hspace{1mm} x \in X_s\}.
	\]
\end{Definition}

An example to keep in mind is a convex polytope in $\mathbb{E}^n$ or $\mathbb{H}^n$ with mirrors the codimension-one faces.
We will assume that our index set $S$ is finite.
\begin{Definition}
	Let $X$ have a mirror structure, and let $W$ be a Coxeter group with generators $s \in S$.
	Let $W_T$ denote the subgroup generated by $s \in T \subset S$.
	Let $\sim$ denote the following equivalence relation on $W \times X: (w_1,x) \sim (w_2,y)$ if and only if $x = y$ and $w_1w_2^{-1} \in W_{S(x)}$.
	The \emph{basic construction} is the space 
	\[
		\mathcal{U}(W,X) := W \times X/ \sim.
	\]
\end{Definition}

Therefore, $\mathcal{U}(W,X)$ is constructed by gluing together copies of $X$ along its mirrors, with the exact gluing dictated by the Coxeter group.
A standard example is where $X$ is a right-angled pentagon in $\mathbb{H}^2$ with mirrors the edges of $X$, and $W$ is the right-angled Coxeter group generated by reflections in these edges.
Then $\mathcal{U}(W,X) \cong \mathbb{H}^2$.

Let $W$ be a Coxeter group with nerve $L$.
Again, $L$ is the simplicial complex with vertex set corresponding to $S$ and simplices corresponding to subsets of $S$ that generate finite subgroups of $W$.
Let $K$ be the cone on the barycentric subdivision of $L$.
$K$ admits a natural mirror structure with $K_s$ the closed star of the vertex corresponding to $s$ in the barycentric subdivision of $L$.
The \emph{Davis complex} $\Sigma(W,S)$ is defined to be the simplicial complex $\mathcal{U}(W,K)$.
\begin{Lemma}
	$\Sigma(W,S)$ has the following properties \cite{d08}.
	\begin{itemize}
		\item $W$ acts properly and cocompactly on $\Sigma(W,S)$ with fundamental domain $K$.
		\item $\Sigma$ admits a cellulation such that the link of every vertex can be identified with $L$.
		Therefore, if $L$ is a triangulation of $S^{n-1}$, then $\Sigma(W,S)$ is an $n$-manifold.
		\item $\Sigma(W,S)$ admits a piecewise Euclidean metric that is $CAT(0)$.
	\end{itemize}
\end{Lemma}

We assume from now on that $W$ is a Coxeter group with nerve a triangulation of $S^{n-1}$.
If $w \in W$ acts as a reflection on $\Sigma(W,S)$, we call the fixed point set a wall, and denote it $\Sigma^w$.
\begin{Lemma}\label{walls}
	Walls in $\Sigma(W,S)$ have the following properties.
	\begin{itemize}
		\item The stabilizer of each wall acts geometrically on the wall.
		\item Each wall and each half-space is a geodesically convex subset of $\Sigma(W,S)$.
		\item The collection of walls separates $\Sigma(W,S)$ into disjoint copies of the fundamental domain $K$.
		\item The stabilizer of each point in $\Sigma(W,S)$ is a finite Coxeter group, and the walls containing that point can be locally identified with the fixed hyperplanes of the standard action of this Coxeter group on $\mathbb{R}^n$.
		\end{itemize}
\end{Lemma}

Though each wall of $\Sigma$ is a contractible submanifold, a $W$-orbit of a wall has in general quite complicated topology.
Even in the simple case where $W$ is generated by reflections in a equilateral triangle in $\mathbb{R}^2$ the $W$-orbit of a wall is not contractible, as $W$-translates of a wall can intersect nontrivially. 
However, passing to suitable subgroup fixes this problem.

\begin{Theorem}\label{hier}
	 $W$ has a finite index torsion-free normal subgroup $\Gamma$, and the action of $\Gamma$ on $\Sigma(W,S)$ admits a hierarchy.
\end{Theorem}
\begin{proof}
	The existence of such a subgroup $\Gamma$ is well-known.
	The cutting submanifolds that we choose will be $\Gamma$-orbits of walls in $\Sigma(W,S)$.
	
	A lemma of Millson and Jaffee in~\cite{m76} shows that any torsion-free normal subgroup of $W$ has the trivial intersection property: for all $\gamma \in \Gamma$, either $\gamma\Sigma^s = \Sigma^s$ or $\gamma\Sigma^s \cap \Sigma^s = \emptyset$.
	Therefore, each $\Gamma$-orbit is a disjoint union of walls and has contractible components.
	
	Once we have removed all the walls, we are left with disjoint copies of the fundamental domain $K$, and since $\Gamma$ is finite index in $W$, there are only finitely many orbits of walls to remove, so by Lemma \ref{walls}, this is a tidy collection.
	Therefore, we are done by Theorem \ref{tidy}.
\end{proof}

\begin{Remark}
If $W$ is a Coxeter group with nerve a triangulation of $D^{n-1}$, then $\Sigma(W,S)$ is an $n$-manifold with boundary, and these groups also virtually admit hierarchies. 
\end{Remark}

\section{$L^2$-homology}\label{s:l2} 

Let $M$ be a $G$-space, and let $C_{*}(M)$ denote the usual cellular chains of $M$, which we regard as left $\mathbb{Z}G$-modules.
The \emph{square-summable chains of $M$} are the tensor product 
\[
	C_{*}^{(2)}(M) = L^{2}(G) \otimes_{\mathbb{Z}G} C_{*}(M)
\]
where $L^{2}(G)$ is the Hilbert space of real-valued square-summable functions on $G$.

The usual boundary homomorphism $\partial: C_{*}(M) \rightarrow C_{{*}-1}(M)$ extends to a boundary operator $\partial: C_{*}^{(2)}(M) \rightarrow C_{{*}-1}^{(2)}(M)$ whose adjoint is the coboundary operator $\delta: C_{*}^{(2)}(M) \rightarrow C_{{*}+1}^{(2)}(M)$.

The $L^{2}$-(co)homology groups can be defined as the kernel of the Laplacian operator: 
\[
	L^2H_{*}(M,G) \cong L^2H^{*}(M,G) \cong \ker (\partial \delta + \delta \partial): C_{*}^{(2)}(M) \rightarrow C_{*}^{(2)}(M).
\]
We shall shorten this to $L^2H_{*}(M)$ if the group action is obvious.
We record as a lemma some of the basic algebraic properties of $L^2$-homology that we will need.
As images of maps are rarely closed subspaces of Hilbert spaces, we must define a \emph{weakly exact sequence} where kernels are equal to the closures of images.
\begin{Lemma}\label{props}
	Let (M,N) be a pair of $G$-spaces.
	\begin{itemize}
		
		\item(Functoriality) If $(M_1,N_1)$ and $(M_2,N_2)$ are pairs of $G$-spaces and $f:(M_1,N_1) \rightarrow (M_2,N_2)$ is a $G$-equivariant map, then there is an induced map $f_{*}: L^2H_k(M_1,N_1) \rightarrow L^2H_k(M_2,N_2)$.
		If $f$ is a $G$-equivariant homotopy equivalence, then $f_{*}$ is an isomorphism.
		
		\item (Exact sequence of a pair) The sequence 
		\[
			\dots \rightarrow L^2H_i(N) \rightarrow L^2H_i(M) \rightarrow L^2H_i(M,N) \rightarrow \dots
		\]
		is weakly exact.
		
		\item (Induction principle) The $L^{2}$-homology of $M$ is induced from the $L^{2}$-homology of its components: 
		\[
			L^2H_i(M;G) = \bigoplus_{[M^{0}]\in\pi_{0}(M)/G} L^2H_i(M^0, \stab_G M^0) \nearrow G,
		\]
		where the sum is over representatives of the orbits of the components of $M$.
				
		\item (Mayer--Vietoris sequences) Suppose $M = M_1 \cup_{M_0} M_2$ and $(M,M_i)$ is a pair of $G$-spaces for $i = 1,2$.
		Then $(M,M_0)$ is a pair of $G$-spaces and the sequence 
		\[
			\dots \rightarrow L^2H_i(M_0) \rightarrow L^2H_i(M_1) \oplus L^2H_i(M_2) \rightarrow L^2H_i(M) \rightarrow \dots
		\]
		is weakly exact.
		
		\item (Poincar\'{e} Duality) If $M$ is a manifold then $L^2H^i(M) \cong L^2H_{n-i}(M, \partial M)$ and $L^2H_i(M) \cong L^2H^{n-i}(M, \partial M)$
	\end{itemize}
\end{Lemma}

We record as a lemma the specific version of the Mayer--Vietoris sequence that we use.
\begin{Lemma}
If $(M,F)$ is a tidy pair and $N$ is $M$ cut-open along $F$, there is a Mayer--Vietoris sequence 
\begin{equation}\label{e:L2MV}
	\rightarrow L^2H_i(F \times \partial I) \rightarrow L^2H_i(F \times  I) \oplus L^2H_i(N) \rightarrow L^2H_i(M) \rightarrow L^2H_{i-1}(F \times \partial I) \rightarrow
\end{equation}
\end{Lemma}

Note that $L^2H_i(F \times \partial I) \cong L^2H_i(F) \oplus L^2H_i(F)$ and $L^2H_i(F \times I) \cong L^2H_i(F)$.
In particular, if any of these terms are zero, the other terms are as well.

\begin{Definition}
	For a discrete group $G, L^2H_i(G)$ is the $L^2$-homology of a contractible $G$-space.
	By the functoriality property, this is well-defined.
\end{Definition}

\section{Vanishing conjectures and results}\label{s:vanish} 

\begin{Conjecture1}
If $G$ acts geometrically on a contractible $n$-manifold without boundary, then $L^2H_i(G) = 0$ for $i \ne n/2$.
\end{Conjecture1}

The conjecture holds for trivial reasons in dimensions $\le 2$. In dimension $3$, Lott and L\"{u}ck~\cite{ll95} proved the conjecture for all fundamental groups of manifolds that satisfy the Geometrization Conjecture, therefore by Perelman~\cite{p02,p03,p03a} it holds for all groups acting geometrically on contractible $3$-manifolds. 
We record this as a theorem.      

\begin{Theorem}\label{t:lowsinger}
The Singer conjecture is true in dimensions $n \le 3$.
\end{Theorem}

\begin{Definition}
	The \emph{action dimension} of a group $G$, $\actdim(G)$ is the least dimension of a contractible manifold which admits a proper $G$-action.
\end{Definition}

Action dimension was introduced and studied by Bestvina, Kapovich, and Kleiner in~\cite{bkk02}.
One consequence of their work is that an $n$-fold product of non-abelian free groups does not act properly discontinuously on a contractible $(2n-1)$-manifold.
Since non-abelian free groups have $L^2H_1(F_n) \ne 0$, it follows from a K{\"u}nneth formula for $L^2$-homology that the $n$-fold products have non-trivial $L^2H_n$.
Therefore, as noted in \cite{do01}, their result is implied by the following conjecture.
\begin{Conjecture2}\label{actdim}
         $L^2H_i(G) = 0$ for $i >\actdim(G)/2$.
\end{Conjecture2}
\begin{Definition}
	The \emph{cocompact action dimension} of a group $G, \cadim(G)$ is the least dimension of a contractible $G$-manifold.
\end{Definition}

\begin{Remark}
We do not know of a group $G$ with $\actdim(G) < \cadim(G)$.
\end{Remark}

\begin{Conjecture3}\label{cadim}
	$L^2H_i(G) = 0$ for $i >\cadim(G)/2$.
\end{Conjecture3}
Since any contractible $G$-manifold can be used to compute $L^2H_i(G)$, we have an equivalent series of conjectures in terms of manifolds.
\begin{Cadimm}\label{cadimm}
	If $(M,\d M)$ is an $n$-manifold with contractible components which admits a geometric group action, then $L^2H_i(M) = 0$ for $i >n/2$.
\end{Cadimm}

\begin{Remark}
	These conjectures put restrictions on the embedding dimension of a $K(G,1)$ space.
	For example, if $L^2H_i(G) \ne 0$, the cadim conjecture implies that any finite $K(G,1)$ space cannot embed in $\mathbb{R}^{2i-1}$.
\end{Remark}
\begin{Lemma}\label{easy}
	$\actdim$ conjecture $\Rightarrow \cadim$ conjecture $\Rightarrow$ Singer conjecture.
\end{Lemma}
\begin{proof}
	The first implication is trivial, and the second follows from applying Poincar\'{e} duality, and the fact that a group acting geometrically on a contractible $n$-manifold without boundary has $\actdim = \cadim = n.$
\end{proof}

If $(M^{2k+1},F^{2k})$ is a tidy $G$-pair, then Poincar\'{e} duality and  the cadim conjecture applied to the action of $G$ on $M$ imply $L^2H_k(M,\partial)=L^2H_{k+1}(M)=0$.
In particular, we have the following apparently weaker version of the cadim conjecture.
\begin{Conjecture5}
	If $(M^{2k+1},F^{2k})$ is a tidy pair, then the map induced by inclusion $i_{*}: L^2H_k(F, \partial) \rightarrow L^2H_k(M,\partial)$ is the zero map.
\end{Conjecture5}

\begin{Lemma}\label{l:induct} 
Suppose that $(M^{n},F)$ is a tidy $G$-pair, $N$ is $M$ cut-open by $F$, and the cadim conjecture holds for  $F$. 
Then the cadim conjecture holds for $M$ if and  only if it holds for $N$ and, if $n=2k+1$ is odd,  the weak cadim conjecture holds for $(M,F)$.
\end{Lemma}	
\begin{proof}
	First, suppose that the cadim conjecture holds for $M$. 
	We have $L^2H_i(M) = 0$ for $i > n/2$, and $L^2H_i(F) = 0$ for $i > (n-1)/2$, so the claim follows from the Mayer--Vietoris sequence~\eqref{e:L2MV}.   
	
	Next, suppose the cadim conjecture holds for $N$, so that we have $L^2H_i(N) = 0$ for $i > n/2$, and $L^2H_i(F) = 0$ for $i > (n-1)/2$. 
	Then the Mayer--Vietoris sequence~\eqref{e:L2MV} implies $L^2H_i(M) = 0$ for $i > (n+1)/2$.
	
	Now, we only have to consider the case where $n=2k+1$ and $i=k+1$. 
	Applying Poincar\'{e} duality to the Mayer--Vietoris sequence gives the following diagram, where the vertical maps are the duality isomorphisms:
\[
\begin{tikzcd}
	L^2H_{k+1}(F \times  I) \oplus L^2H_{k+1}(N) \dar \rar &L^2H_{k+1}(M) \rar{\partial_{*}} \dar & L^2H_k(F \times \partial I)  \dar  \\
	L^2H^{k}(F \times  I,\partial) \oplus L^2H^{k}(N,\partial) \uar{\cong} \rar & L^2H^k(M,\partial) \uar{\cong} \rar{i^{*}}  & L^2H^k(F \times \partial I, \partial)  \uar{\cong}
\end{tikzcd}
\]
Since $F \times \partial I$ is just two copies of $F$, the weak cadim conjecture implies that $i^{*}=0$ and the result follows. 
\end{proof}

\begin{Theorem}\label{induct}
The cadim conjecture in dimension $2k-1$ implies the cadim conjecture in dimension $2k$ for manifolds with hierarchies. 
The cadim conjecture in dimension $2k$ and the weak cadim conjecture in dimension $2k+1$  imply the cadim conjecture in dimension $2k+1$ for manifolds with hierarchies.	
\end{Theorem}
\begin{proof}
This is immediate by induction on the length of the hierarchy, using Lemmas~\ref{l:stab} and~\ref{l:induct}, and noting that the cadim conjecture holds for manifolds with compact components.
\end{proof}	

A somewhat surprising result is the converse to the second implication in Lemma~\ref{easy}.
\begin{Theorem}\label{equiv}
The Singer conjecture and the cadim conjecture are equivalent.
\end{Theorem}
The proof of the theorem follows immediately from the following key lemma and induction.
\begin{Lemma}\label{l:key}
The Singer conjecture in dimension $n$ and the cadim conjecture in dimension $(n-1)$ imply the cadim conjecture in dimension $n$.
\end{Lemma}
\begin{proof}
We use the equivariant Davis reflection group trick as in \cites{dl03a, d08}.
The idea is that the trick turns the input of the cadim conjecture (a contractible manifold with boundary and geometric group action) into the input of the Singer conjecture (a contractible manifold without boundary and geometric group action).
In addition, the newly constructed manifold action admits a hierarchy ending at a disjoint union of copies of the original.
Once this has been established, the proof is more or less the same as that of Theorem \ref{induct}.

	Suppose that $G$ acts geometrically on a contractible $n$-manifold with boundary $(M, \partial M)$.
	Let $L$ be a flag triangulation of $\partial M$ that is equivariant with respect to the $G$-action, and suppose that the stabilizer of any simplex fixes the stabilizer pointwise (these triangulations can always be constructed).
	We can now apply the equivariant reflection group trick.
	Indeed, $L$ determines a right-angled Coxeter group $W$, and we can form the basic construction $\mathcal{U} = \mathcal{U}(W,M)$.
	By the conditions imposed on $L$, there is an action of $G$ on $W$ which determines a semi-direct product $W \rtimes G$.
	Since  $ \mathcal{U}/W \rtimes G \cong M/G$, $W \rtimes G$ acts cocompactly on $\mathcal{U}$.	
	Here are some key properties of the reflection group trick: 
	\begin{itemize}
		\item Each wall is a codimension one, contractible submanifold of $N$.
		\item There are a finite number of $W \rtimes G$-orbits of walls, and each orbit is a disjoint union of walls.
		\item Any non-empty intersection of orbits of walls is itself a Davis complex and is therefore contractible.
		\item The stabilizer of each wall acts geometrically on the wall.
		\item The collection of walls looks locally like a right-angled hyperplane arrangement in $\mathbb{R}^n$.
	\end{itemize}
		
	It follows similarly to Theorem~\ref{tidy} that the $W \rtimes G$ action on $\mathcal{U}$ admits a hierarchy that ends in disjoint copies of $M$, where the cutting submanifolds are $W \rtimes G$-orbits of walls.
	Since $\cU$ has no boundary, and we are assuming that the Singer conjecture holds for $\mathcal{U}$,   the cadim conjecture holds for $\cU$.
	Since we are also assuming the cadim conjecture in dimension $n-1$, it follows by applying Lemma~\ref{l:induct}	inductively that the cadim conjecture holds for the original $M$.
\end{proof}

Theorem~\ref{t:lowsinger}  and Lemma~\ref{l:key} imply:
\begin{Corollary}\label{c:lowcadim}
The cadim conjecture holds true in dimensions $\le 3$. 
\end{Corollary}

 Now, Corollary \ref{c:lowcadim}, Theorem~\ref{induct} and Lemma~\ref{easy} imply our main theorem.

\begin{Theorem}\label{hierarchy}
The Singer conjecture holds for all groups that admit a hierarchy in dimensions $\le 4$.
\end{Theorem}

Theorem \ref{hierarchy} and Theorem \ref{hier} now imply our main applications:
\begin{Theorem}
	If $W$ is a Coxeter group with nerve a triangulation of $S^3$, then the Singer conjecture holds for $W$ acting on $\Sigma(W,S)$.\end{Theorem}
	
\begin{Theorem}
 If $W$ is a Coxeter group with nerve a triangulation of $D^3$, then $L^2H_i(W) = 0$ for $i > 2$.
\end{Theorem}

\begin{Remark}
The hierarchies for Coxeter groups have more structure in the following sense: the hierarchy for $\Sigma(W,S)$ induces a hierarchy on each wall. 
This means that if we restrict our attention to Coxeter groups, we can relax many of the assumptions. 
For instance, Theorem \ref{induct} restricted to Coxeter groups need only assume the $\cadim$ conjecture in dimension $2k-1$ for manifolds with hierarchies.
\end{Remark}

\end{document}